\documentclass[10pt]{article}

\usepackage{amsmath,amssymb,amsthm}
\usepackage[usenames,dvipsnames]{xcolor} 	
\usepackage{enumerate} 
\usepackage{fullpage}

\numberwithin{equation}{section}

\theoremstyle{plain}
\newtheorem{theorem}{Theorem}[section]
\newtheorem{proposition}[theorem]{Proposition}

\newtheorem{lemma}[theorem]{Lemma}

\theoremstyle{definition}
\newtheorem{definition}[theorem]{Definition}
\newtheorem{convention}[theorem]{Convention}

\theoremstyle{remark}
\newtheorem{remark}[theorem]{Remark}

\DeclareMathOperator{\adj}{adj} 
\newcommand{\aplim}{\operatorname*{ap\:lim}}

\DeclareMathOperator{\cof}{cof}
\DeclareMathOperator{\Det}{Det}
\DeclareMathOperator{\dist}{dist}       
\renewcommand{\div}{\operatorname{div}}
\DeclareMathOperator{\Div}{Div}     
\DeclareMathOperator{\imG}{im_G}
\DeclareMathOperator{\imT}{im_T}
\DeclareMathOperator{\loc}{loc}

\DeclareMathOperator{\spt}{spt}
\DeclareMathOperator{\mec}{mec}

\DeclareMathOperator*{\essosc}{ess \: osc}
\DeclareMathOperator{\diam}{diam}
\DeclareMathOperator{\inv}{in}

\renewcommand{\O}{\Omega}
\newcommand{\id}{\mathbf{id}}         
\newcommand{\dd}{\mathrm{d}}
\newcommand{\N}{\mathbb{N}}
\newcommand{\Z}{\mathbb{Z}}

\newcommand{\R}{\mathbb{R}} 
\newcommand{\Rn}{\mathbb{R}^n} 
\newcommand{\Rnn}{\mathbb{R}^{n \times n}}            
\newcommand{\res}{\mathop{\hbox{\vrule height 7pt width .5pt depth 0pt \vrule height .5pt width 6pt depth 0pt}}\nolimits\,}
\newcommand{\ssubset}{\subset \! \subset}
\newcommand{\Sn}{\mathbb{S}^{n-1}}
\renewcommand{\vec}{\mathbf}    
\newcommand{\vecg}{\boldsymbol}
\newcommand{\weakc}{\rightharpoonup}
\newcommand{\weakcs}{\overset{*}{\rightharpoonup}}

\newcommand{\mc}{\mathcal}
\newcommand{\p}{\partial}

\renewcommand{\d}{\delta}

\newcommand{\e}{\varepsilon}

\newcommand{\f}{\varphi}

\newcommand{\Ln}{\mathcal{L}^n}
\newcommand{\Hn}{\mathcal{H}^{n-1}}
\newcommand{\vacio}{\varnothing}

\def\Xint#1{\mathchoice
{\XXint\displaystyle\textstyle{#1}}%
{\XXint\textstyle\scriptstyle{#1}}%
{\XXint\scriptstyle\scriptscriptstyle{#1}}%
{\XXint\scriptscriptstyle\scriptscriptstyle{#1}}%
\!\int}
\def\XXint#1#2#3{{\setbox0=\hbox{$#1{#2#3}{\int}$}
\vcenter{\hbox{$#2#3$}}\kern-.5\wd0}}

\def\dashint{\Xint-}

\title{Orlicz-Sobolev nematic elastomers}\date{}

\author{Duvan Henao and Bianca Stroffolini }
\newcommand{\Addresses}{{
  \bigskip
  \footnotesize

 B.~Stroffolini, \textsc{Dipartimento di Ingegneria Elettrica e delle Tecnologie dell'Informazione , Universit\`{a} di Napoli Federico II, Via Claudio, 80126 Napoli, Italy}\par\nopagebreak
  \textit{E-mail address}, B.~Stroffolini: \texttt{bstroffo@unina.it}

  \medskip

D..~Henao, \textsc{Faculty of Mathematics , Pontificia Universidad Cat\'olica de Chile, Vicu\~na Mackenna 4860, Macul, Santiago, Chile}\par\nopagebreak
  \textit{E-mail address}, D.~Henao: \texttt{dhenao@mat.puc.cl}

}}

\begin{document}

\maketitle

\begin{abstract} 
  We extend  the existence theorems
  in [Barchiesi, Henao \& Mora-Corral; ARMA 224],
  for models of nematic elastomers and magnetoelasticity,
  to a larger class in the scale of Orlicz spaces.
  These models consider both an elastic term where 
  a polyconvex energy density is composed with
  an unknown state variable defined in the deformed configuration,
  and a functional corresponding to the nematic energy (or the exchange and magnetostatic energies
  in magnetoelasticity)
  where the energy density is integrated over the deformed configuration.
  In order to obtain the desired compactness and lower semicontinuity,
  we show that the regularity requirement that maps create no new surface 
  can still be imposed when the gradients are in an Orlicz class with an integrability
  just above the space dimension minus one.
  We prove that the fine properties of orientation-preserving maps satisfying that regularity requirement
  (namely, being weakly 1-pseudomonotone, $\mathcal H^1$-continuous, a.e.\ differentiable, and a.e.\ locally invertible)
  are still valid in the Orlicz-Sobolev setting. 
\end{abstract}


\section{Introduction}

Motivated by the modelling of nematic elastomers, Barchiesi \& DeSimone \cite{Barchiesi15}
analyzed the minimization of functionals of the form 
\begin{align}
 \label{eq:model1}
 I(\vec u, \vec n)= \int_\Omega W_{\mec}(D\vec u(\vec x), \vec n(\vec u(\vec x)))\dd\vec x 
  + \int_{\vec u(\Omega)} |D\vec n(\vec y)|^2 \dd\vec y
\end{align}
where $\Omega\subset \R^3$, $\vec u\in W^{1,p}(\Omega, \R^3)$ for some $p>3$,
$\vec n \in H^1(\vec u(\Omega), \R^2)$, and 
\begin{align}
 \label{eq:model2}
  W_{\mec}(\vec F, \vec n)=  W \Big ( (\alpha^{-1}\vec n \otimes \vec n + \sqrt{\alpha} (I-\vec n
\otimes \vec n)) \vec F\Big )
\end{align}
for a certain $\alpha>0$ and some polyconvex energy function $W$. 
Functionals with a similar structure appear also 
in models describing the nematic mesogens with the Landau-de Gennes theory, 
and 
in magnetoelasticity and plasticity, see, e.g., 
\cite{CaGaL14,DaFo92,FoGa95paper,Kruzik15,BHM17}.
The major difficulties are that $I$ depends on the composition of the two unknowns and
that the nematic director $\vec n$ is defined in the domain $\vec u(\Omega)$ 
which is also determined only as a part of the solution of the variational problem. 
The analysis is based on the inverse function theorem for Sobolev maps due to 
Fonseca \& Gangbo \cite{FoGa95paper}, which is valid for $W^{1,p}$ maps from a domain in $\R^n$
to $\R^n$ when $p>n$. Using the results for the Sobolev regularity of the inverse obtained in
\cite{HeMo10,HeMo11,HeMo12,HeMo15}, both the local invertibility theorem of Fonseca \& Gangbo
and the analysis of Barchiesi \& DeSimone were generalized by Barchiesi, Henao \& Mora-Corral 
\cite{BHM17} to a suitable class of maps in $W^{1,p}(\Omega, \R^3)$ 
for all $p>2$.
The importance of relaxing the hypothesis on the integrability exponent $p$
is that, on the one hand, they are related to the coercivity that the stored energy function $W$ 
is assumed to possess and, on the other hand, 
the analysis should ideally depend as little as possible 
on the behaviour of $W$ at infinity (for physical reasons).
Here the less restrictive condition that 
\begin{equation} \label{eq:young-1}
\int_\Omega A(|D\vec u(\vec x)|)\dd\vec x <\infty 
\end{equation}
for some Young function $A:[0,\infty)\to [0,\infty]$ 
satisfying 
\begin{equation}\label{young}
\int^\infty \frac{t}{A(t)} \dd t <\infty
\end{equation}
(e.g.\ $A(t):=t^2\log^\alpha t$ for any $\alpha>1$)
is shown to be also sufficient to establish the existence of minimizers
for functionals like $I(\vec u, \vec n)$ in \eqref{eq:model1}.

In the paper \cite{Kauhanen99}, the authors investigated the minimal analytic assumptions 
on a map $u\colon\Omega \to \mathbb{R}^n$ to guarantee continuity, 
differentiability a.e. and the Lusin (N) condition. As far as the condition (N) is concerned,
the  $n$-absolute continuity introduced by Mal\'y in \cite{Maly99} plays an important role. 
It turned out that this condition is satisfied by a function $u\in W^{1,1}(\Omega)$ whenever 
their weak partial derivatives are in the Lorentz space $L^{n,1}(\Omega)$. 
In particular, they characterize the space $L^{n,1}$ in terms of an Orlicz integrability condition. 
This condition is exactly the one stated in \cite {Cianchi96}, see Theorem \ref{pr:embedding}.
We will prove this condition on manifolds of dimension $n-1$.

Our result, on the one hand, enlarges the class of maps in which the minimization problem can be set.
On the other hand, it sheds new light on results on invertibility of maps and interpenetration of matter.
In fact, we can consider the class of Sobolev-Orlicz maps
and define accordingly the notion of zero surface energy ($\mathcal{E}(\vec u)=0$, see Definition \ref{def:SE}). 
This, in turn, when imposed together with the positivity of the Jacobian determinant,
is equivalent to the requirement that 
$\Det D\vec u = \det D\vec u$ 
(where $\Det D\vec u$ denotes the distributional determinant, see Definition \ref{de:Det})
and that $\vec u$ preserves orientation in the topological sense.
\begin{theorem}Let $A$ be a Young function satisfying \eqref{young} and 
suppose that $\vec u\in W^{1,A}(\Omega. \R^n)$ satisfies $det D\vec u\in L^1_{loc}(\Omega)$, Then we have the equivalence:
\begin{itemize}
\item{}$\mathcal{E}(\vec u)=0$ and $detD\vec u>0$ a.e.; 
\item{}$(\adj D\vec u)\vec u\in L^1_{loc}(\Omega. \R^n)$, $\det D\vec u(x)\not=0$ for a.e.\ $x\in \Omega$, 
  $\det D\vec u=\Det D\vec u$ and $\deg (\vec u, B(\vec x, r))\ge 0$ for every $\vec x\in \Omega$ and a.e.\ $r\in (0, \dist(\vec x, 
  \partial \Omega))$.
\end{itemize}
\end{theorem}

This article explains the new ideas and the results in the literature of Orlicz-Sobolev spaces
that are required to generalize the analysis of  \cite{BHM17} (full detail of the proofs is not given 
since that would render the article unnecessarily long, given the technical difficulties).
Section \ref{se:notation} is for notation and preliminaries. Section \ref{se:H1continuity}
proves that weakly monotone 
maps  having the integrability \eqref{eq:young-1}-\eqref{young} 
are continuous 
at every point outside an $\mathcal H^1$-null set 
(in the classical sense, not only in the sense of quasi-continuity). 
The functional class of orientation-preserving Orlicz-Sobolev maps creating no surface,
proposed for the modelling of nematic elastomers, is defined and studied in Section \ref{se:classA}.
Concretely, maps in this class are proved to be $1$-pseudomonotone, \cite{HaMa02};
 to have a precise representative that satisfies Lusin's contition and is $\mathcal H^1$-continuous
and a.e.\ differentiable;  to be, in a certain sense, open and proper; 
and to be locally invertible around almost every point, the local inverses
and their minors 
being Sobolev and sequentially weakly continuous.
The main existence theorem, for functionals, such as \eqref{eq:model1}, 
defined both in the reference and in the deformed configuration,
is stated finally in Section \ref{se:lower}.

\section{Notation and preliminaries}\label{se:notation}

\subsection{General notation}
We will work in dimension $n \geq 3$, and $\O$ is a bounded open set of $\Rn$.
Vector-valued and matrix-valued quantities will be written in boldface.
Coordinates in the reference configuration will be denoted by $\vec x$, and in the deformed configuration by $\vec y$.

The characteristic function of a set $A$ is denoted by $\chi_A$.
Given two sets $U,V$ of $\Rn$, we will write $U \ssubset V$ if $U$ is bounded and $\bar{U} \subset V$.
The open ball of radius $r>0$ centred at $\vec x \in \Rn$ is denoted by $B(\vec x, r)$; 
unless otherwise stated, a \emph{ball} is understood to be open.
The $(n-1)$-dimensional sphere in $\Rn$ centred at $\vec x_0$, with radius $r$, is denoted by 
$S(\vec x_0, r)$ or $S_r(\vec x_0)$. 

Given a square matrix $\vec A \in \Rnn$, 
the adjugate matrix $\adj \vec A$ satisfies $(\det \vec A) \vec{I} = \vec A \adj \vec A$, where $\vec{I}$ denotes the identity matrix.
The transpose of $\adj \vec A$ is the cofactor $\cof \vec A$.
If $\vec A$ is invertible, its inverse is denoted by $\vec A^{-1}$.
The inner (dot) product of vectors and of matrices will be denoted by $\cdot$.
The Euclidean norm of a vector $\vec x$ is denoted by $|\vec x|$, and the associated matrix norm is also denoted by $\left| \cdot \right|$.
Given $\vec a, \vec b \in \Rn$, the tensor product $\vec a \otimes \vec b$ is the $n \times n$ matrix whose component $(i,j)$ 
is $a_i \, b_j$.
The set $\Rnn_+$ denotes the subset of matrices in $\Rnn$ with positive determinant.

The Lebesgue measure in $\Rn$ is denoted by $\mc{L}^n$, and the 
$(n-1)$-dimensional Hausdorff measure by $\mc{H}^{n-1}$.
The abbreviation \emph{a.e.} stands for \emph{almost everywhere} or \emph{almost every}; 
unless otherwise stated, it refers to the Lebesgue $\Ln$ measure.
For $1 \leq p \leq \infty$, the Lebesgue $L^p$, Sobolev $W^{1,p}$ and 
bounded variation $BV$ spaces are defined in the usual way.
So are the functions of class $C^k$, 
for $k$ a positive integer of infinity, and their versions $C^k_c$ of compact support.
The set of (positive or vector-valued) Radon measures is denoted by $\mc{M}$.
The conjugate exponent of $p$ is written $p'$.
We do not identify functions that coincide a.e.;
moreover an $L^p$ or $W^{1,p}$ function may eventually be defined only at a.e.\ point of its domain.
We will indicate the domain and target space, as in, for example, $L^p (\O,\Rn)$,
except if the target space is $\R$, in which case we will simply write $L^p (\O)$.
Given $S \subset \Rn$, the space $L^p (\O,S)$ denotes the set of $\vec u \in L^p (\O,\Rn)$ 
such that $\vec u (\vec x) \in S$ for a.e.\ $\vec x \in \O$.
The space $W^{1,p}_{\loc} (\O)$ is the set of funcions $\vec u$ defined in $\O$ such that $\vec u |_A \in W^{1,p} (A)$ for any open $A \ssubset \O$; we will analogously use the subscript $\loc$ for other function spaces.
Weak convergence (typically, in $L^p$ or $W^{1,p}$) is indicated by $\weakc$,
while $\weakcs$ is the symbol for weak$^*$ convergence in $\mc{M}$ or in $BV$.
The supremum norm in a set $A$ (typically, a sphere) is indicated by $\left\| \cdot \right\|_{\infty,A}$, 
while $\dashint_A$ denotes the integral in $A$ divided by the measure of $A$.
The identity function in $\Rn$ is denoted by $\id$.
The support of a function is indicated by $\spt$.

The distributional derivative of a Sobolev function $\vec u$ is written $D \vec u$, which is defined a.e\@.
If $\vec u$ is differentiable at $\vec x$, its derivative is denoted by $D \vec u (\vec x)$, 
while if $\vec u$ is differentiable everywhere, the derivative function is also denoted by $D \vec u$.
Other notions of differentiability, 
which carry different notations, are explained in Section \ref{subse:approximate} below.

If $\mu$ is a measure on a set $U$, and $V$ is a $\mu$-measurable subset of $U$, then the restriction of $\mu$ to $V$ is  
denoted by $\mu \res V$. The measure $|\mu|$ denotes the total variation of $\mu$. 

Given two sets $A, B$ of $\Rn$, we write $A \subset B$ a.e.\ 
if $\mc{L}^n (A \setminus B) = 0$, while $A = B$ a.e.\ means 
$A \subset B$ a.e.\ and $B \subset A$ a.e\@.
An analogous meaning is given to the expression $\Hn$-a.e\@. 
With $\bigtriangleup$ we denote the symmetric difference of sets: 
$A \bigtriangleup B := (A \setminus B) \cup (B \setminus A)$.

In the proofs of convergence, we will continuously use subsequences, which not be relabelled.
 \subsection{Orlicz-Sobolev spaces}
 
 We follow the presentation in \cite{Carozza19}
 and refer the reader to \cite{KrRu61,Rao91,Rao02} for a comprehensive treatment.
 A function $A:[0,\infty)\to[0,\infty]$ is called a Young function if it is convex, non constant in $(0,\infty)$,
 and vanishes at $0$. Any function fulfilling these properties has the form 
 \begin{align}
  A(t)=\int_0^t a(r)\dd r\quad \text{for}\ t\geq 0,
 \end{align}
 for some non-decreasing, left-continuous function $a:[0,\infty)\to[0,\infty]$ which is neither identically $0$
 nor infinity. The function 
 \begin{align}
   \label{eq:Attincreases}
  t\mapsto \frac{A(t)}{t}\quad \text{is non-decreasing,}
 \end{align}
 and
 \begin{align}
  A(t)\leq a(t)t\leq A(2t)\quad \text{for}\ t\geq 0.
 \end{align}
 A Young function $A$ is said to satisfy the $\Delta_2$-condition near infinity
 if it is finite-valued and there exist constants $C>2$ and $t_0>0$ such that
 \begin{align}
  A(2t)\leq CA(t)\quad \text{for}\ t\geq t_0.
 \end{align}
 
 The Young conjugate $\widetilde A$ of $A$ is defined by 
 \begin{align}
  \widetilde A(t)=\sup\{st-A(s):s\geq 0\}\quad \text{for}\ t\geq 0.
 \end{align}
 It is known that $\widetilde{{\widetilde A}}=A$.

 An $N$-function $A$ is a convex function from $[0,\infty)$ into $[0,\infty)$ which vanishes 
 only at $0$ and such that $\lim_{s\to 0^+} \frac{A(s)}{s}=0$ and $\lim_{s\to \infty} \frac{A(s)}{s}=\infty.$

 Let $E$ be a measurable subset of $\R^n$.
 The Orlicz space $L^A(E)$ built upon a Young function $A$ is the Banach function space of those 
 real-valued measurable functions $u$ on $E$ for which the Luxemburg norm 
 \begin{align*}
  \|u\|_{L^A(E)} = \inf \left \{ \lambda>0: \int_E A\left ( \frac{|u|}{\lambda}\right ) \dd\vec x 
  \leq 1 \right \}
 \end{align*}
 is finite. Since $A$ is non-decreasing,
 \begin{align}
  \int_E A(|u|)\dd\vec x <\infty \ \Rightarrow\ \|u\|_{L^A(E)}\leq 1.
 \end{align}
 If $A$ satisfies the $\Delta_2$-condition at infinity then 
 \begin{align} \label{eq:slicing1}
  u\in L^A(E)\ \Leftrightarrow\ \int_E A(|u|)\dd\vec x <\infty.
 \end{align}
 
 Given an open set $\Omega\subset \R^n$ and a Young function $A$, the Orlicz-Sobolev space $W^{1,A}(\Omega)$
 is defined as 
 $$ W^{1,A}(\Omega) = \{u \in L^A(\Omega): u\ \text{is weakly differentiable, and}\ 
  |\nabla u| \in L^A(\Omega)\}.$$
 The space $W^{1,A}(\Omega)$, equipped with the norm given by 
 $$\|u\|_{W^{1,A}(\Omega)} = \|u\|_{L^A(\Omega)} + \|\nabla u\|_{L^A(\Omega, \R^n)}$$
 for $u\in W^{1,A}(\Omega)$, is a Banach space.
 
 \subsection{Lorentz spaces}
 
 Given a measure space $(X,\mu)$ and $1\leq q<p<\infty$,
 the \emph{distribution function} of a measurable function $u$ on $X$
 is defined by 
 \begin{align*}
  \omega(\alpha,u)=\mu(\{x\in X: |u(x)|>\alpha\}),\quad \alpha\geq 0.
 \end{align*}
 The \emph{nonincreasing rearrangement} $u^*$ of $u$ is defined by 
 \begin{align*}
  u^*(t)=\inf \{\alpha\geq 0: \omega(\alpha, u)\leq t\}.
 \end{align*}
 The Lorentz space $L^{p,q}(X)$ is defined as the class of all measurable functions on $X$
 for which the norm 
 \begin{align*}
  \|u\|_{L^{p,q}(X)}:= \left ( \int_0^{\mu(X)} (t^{1/p}u^*(t))^q\frac{\dd t}{t}\right )^{1/q}
 \end{align*}
 is finite. For more on Lorentz spaces see, e.g.\ \cite{Stein71}.

 \subsection{Approximate differentiability and geometric image} \label{subse:approximate}

 The \emph{density} $D(E,\vec x)$ of a measurable set 
$E \subset \Rn$ at an $\vec x \in \Rn$ is defined as
\[
 D(E,\vec x) := \lim_{r \searrow 0} \frac{\mc{L}^n (E \cap B(\vec x, r))}{\mc{L}^n (B(\vec x, r))} .
\]
 
The following notions are due to Federer \cite{Federer69} 
(see also \cite[Def.\ 2.3]{MuSp95} or \cite[Def.\ 4.31]{AmFuPa00}).

\begin{definition}\label{de:wapproximate}
Let $\vec u : \O \to \Rn$ be measurable function, and consider $\vec x_0 \in \O$.
\begin{enumerate}[a)]	
\item We say that the approximate limit of $\vec u$ at $\vec x_0$ is $\vec y_0$ when
\begin{align*}
	D\big (\{\vec x \in \Omega: |\vec u(\vec x)-\vec y_0|\geq \delta\},\vec x_0)=0
	\quad \text{for each}\ \delta>0.
\end{align*}
In this case, we write $\aplim_{\vec x\to\vec x_0} \vec u(\vec x)=\vec y]_0$.
We say that $\vec u$ is approximately continuous at $\vec x_0$ 
if $\vec u$ is defined at $\vec x_0$ and $\aplim_{\vec x\to\vec x_0} \vec u(\vec x)=\vec u(\vec x_0)$.

\item\label{item:AD} We say that $\vec u$ is approximately differentiable at $\vec x_0$ 
if $\vec u$ is approximately continuous at $\vec x_0$ and there exists $\vec F \in \Rnn$ such that
\[
  D \left( \left\{ \vec x \in \Omega \setminus \{ \vec x_0 \} 
: \frac{|\vec u (\vec x) - \vec u (\vec x_0) - \vec F (\vec x - \vec x_0)|}{|\vec x - \vec x_0|} \geq \d \right\} , \vec x_0 \right) = 0 \quad \text{for each}\ \delta>0.
\]
In this case, $\vec F$ is uniquely determined, 
called the approximate differential of $\vec u$ at $\vec x_0$, and denoted by $\nabla \vec u (\vec x_0)$.
\item	
We denote  the set of approximate differentiability points of $\vec u$ by $\O_d$, or, when we want to emphasize 
the dependence on $\vec u$, by $\O_{\vec u,d}$.
\end{enumerate}
\end{definition}

Given a measurable $\vec u : \O \to \Rn$ that is approximately differentiable a.e., for any $E \subset \Rn$ 
and $\vec y \in \Rn$, we denote by $\mc{N}_E (\vec y)$ the number of $\vec x \in \O_d \cap E$ 
such that $\vec u (\vec x) = \vec y$.
We will use the following version of Federer's \cite{Federer69} area formula, the formulation of which is taken 
from \cite[Prop.\ 2.6]{MuSp95}.

\begin{proposition}\label{pr:areaformula}
Let $\vec u : \O \to \Rn$ be measurable, approximately differentiable a.e\@.
Then, for any measurable set $E \subset \O$ and any measurable function $\f : \Rn \to \R$,
\[
 \int_E \f (\vec u (\vec x)) \left| \det D \vec u (\vec x) \right| \dd \vec x 
= \int_{\Rn} \f (\vec y) \, \mc{N}_E (\vec y) \, \dd \vec y,
\]
whenever either integral exists. 
Moreover, given $\psi: E \to \R$ measurable, the function $\bar{\psi}: \vec u (\Omega_d \cap E) \to \R$ defined by
\[
 \bar{\psi}(\vec y) := \sum_{\substack{\vec x \in \Omega_d \cap E \\ \vec u (\vec x) = \vec y}} \psi (\vec x)
\]
is measurable and satisfies
\[
 \int_E \psi (\vec x) \, \f (\vec u (\vec x)) \left| \det D \vec u (\vec x) \right| \dd \vec x 
= \int_{\vec u(\O_d \cap E)} \bar{\psi}(\vec y) \, \f (\vec y) \, \dd \vec y ,
\]
whenever the integral of the left-hand side exists.
\end{proposition}

We recall the definition of a.e.\ invertibility.
\begin{definition}
 \label{df:1-1ae}
	A function $\vec u:\Omega\to\R^n$ is said to be one-to-one a.e.\ in
a subset $E$ of $\Omega$
if there exists an $\mathcal L^n$-null subset $N$ of $E$ such that
$\vec u|_{E\setminus N}$ is one-to-one.
\end{definition}

Now we present the notion of the \emph{geometric image} of a set (see \cite{MuSp95,CoDeLe03,HeMo12}) in the context of Orlicz spaces.

\begin{definition}\label{de:O0}
Let $\vec u \in W^{1,A} (\O, \Rn)$ and suppose that $\det D\vec u (\vec x)\neq 0$ for a.e.\ $\vec x \in \O$.
Define $\O_0$ as the set of $\vec x \in \O$ for which the following are satisfied:
\begin{enumerate}[a)]
\item\label{item:1O0} $\vec u$ is approximately differentiable at $\vec x$ and
$\det \nabla \vec u (\vec x) \ne  0$; and
\item\label{item:2O0} there exist $\vec w \in C^1 (\Rn, \Rn)$ and a compact set $K \subset \O$ of density $1$ at $\vec x$ such 
that $\vec u|_K = \vec w|_K$ and $\nabla \vec u|_K = D \vec w|_K$.
\end{enumerate}
In order to emphasise the dependence on $\vec u$, the notation $\O_{\vec u, 0}$ will also be employed.
For any measurable set $E$ of $\O$, we define the geometric image of $E$ under $\vec u$ as $\vec u (E \cap \O_0)$, 
and denote it by $\imG (\vec u, E)$.
\end{definition}

The set $\O_0$ is of full measure in $\O$.
Indeed, the Calder\'on--Zygmund theorem shows that property \emph{\ref{item:1O0})} is satisfied a.e., while 
standard arguments, essentially due to Federer \cite[Thms.\ 3.1.8 and 3.1.16]{Federer69} (see also \cite[Prop. 2.4]{MuSp95} 
and \cite[Rk.\ 2.5]{CoDeLe03}), show that property \emph{\ref{item:2O0})} is also satisfied a.e.  
Note also that $\vec u$ is well defined at every $\vec x\in \Omega_0$, because of Definition 
\ref{de:wapproximate}\,\emph{\ref{item:AD})}.

We present the notion of tangential approximate differentiability (cf.\ \cite[Def.\ 3.2.16]{Federer69}).

\begin{definition}
Let $S \subset \R^n$ be a $C^1$ differentiable manifold of dimension $n-1$, and let $\vec x_0 \in S$.
Let $T_{\vec x_0} S$ be the linear tangent space of $S$ at $\vec x_0$.
A map $\vec u: S \to \R^n$ is said to be $\mathcal H^{n-1}\res S$-approximately differentiable at $\vec x_0$ 
if there exists $\vec L \in \Rnn$ such that for all $\d >0$,
\[
\lim_{r \searrow 0} \frac{1}{r^{n-1}} \, \mathcal{H}^{n-1} \left( \Big\{ \vec x \in S \cap B (\vec x_0, r) 
: \frac{|\vec u (\vec x) - \vec u (\vec x_0) - \vec L (\vec x - \vec x_0)|}{|\vec x - \vec x_0|} \geq \delta \Big\} \right) = 0 .
\]
In this case, the linear map $\vec L |_{T_{\vec x_0} S}: T_{\vec x_0} S\to \Rn$ is uniquely determined, called the tangential 
approximate derivative of $\vec u$ at $\vec x_0$, and is denoted by $\nabla \vec u (\vec x_0)$.
\end{definition}

 \subsection{Growth at infinity, continuity and Lusin's condition}
 
 The focus of this paper is on functions $A$ whose growth at infinity is at least such that
 \begin{align} \label{eq:L2logL}
 \int^\infty \left ( \frac{t}{A(t)}\right)^\frac{1}{n-2}\dd t < \infty.
 \end{align}
 The condition is satisfied, in particular, when $A(t)=t^p$ for every $p>n-1$
 and when $A(t)=t^{n-1}\log^\alpha t$ for every $\alpha>n-2$.
 
 Orlicz spaces are intermediate between $L^p$ spaces. In particular,  $L^{n-1}$ contains $L^{A}$
 for any $A$ satisfying \eqref{eq:L2logL}
 (see \cite{Peetre70} or \cite{Kufner}).
 
 As pointed out in \cite[Rmk.\ 3.2]{Carozza19}, condition \eqref{eq:L2logL} is enough to ensure that 
 maps defined on $(n-1)$-dimensional $C^1$ manifolds 
 and having $W^{1,A}$ regularity necessarily have a continuous representative
 and belong to the Lorentz space $L^{n-1,1}$.
 
 \begin{proposition}
   \label{pr:embedding}
  Let $S\subset \R^n$ be a $C^1$ differentiable manifold of dimension $n-1$.
  If an $N$-function $A$ satisfies \eqref{eq:L2logL} and the $\Delta_2$-condition 
  at infinity then every $u\in W^{1,A}(S)$
  has a continuous representative and $Du$
  is of class $L^{n-1,1}$.
  Moreover, 
  there exists a constant $C$, depending only on $A$, $S$, and $n$,
  such that $$\left \|u - \dashint_S u\, \dd\mathcal H^{n-1}\right \|_{L^\infty} \leq C\|Du\|_{L^A(S)}.$$ 
 \end{proposition}
 
 \begin{proof}
  Using local charts $S$ may be assumed, without loss of generality, to be a bounded open subset of $\R^{n-1}$.
  The embedding into $C(S)$ is proved in \cite[Thm.\ 1b]{Cianchi96} under the assumption that
  \begin{align}
    \label{eq:growth2}
   \int^\infty \frac{\widetilde A(t)}{t^{1+m'}}\dd t <\infty,
  \end{align}
  with $m=n-1$. By \cite[Lemma 2.3]{Cianchi04} applied to $\widetilde A$ and $q=m'$ (taking into account
  that $\widetilde{\widetilde A}=A$), condition \eqref{eq:growth2} is equivalent to \eqref{eq:L2logL}.
  
  Define 
  $$\varphi(t):=\left ( \frac{t}{A(t)} \right )^{\frac{n-1}{n-2}}.$$
  Note that $\varphi$ is non-increasing because of \eqref{eq:Attincreases}.
  Also, 
  $$\int_0^\infty \varphi^{\frac{1}{n-1}}(t)\dd t 
   = \int_0^\infty \left ( \frac{t}{A(t)}\right)^\frac{1}{n-2}\dd t$$
   and
   $$
   \int_{|Du|>0} |Du(\vec x)|\varphi^{\frac{1}{n-1}-1}(|Du(\vec x)|)\dd\vec x 
   =\int_{|Du|>0} A(|Du(\vec x)|)\dd\vec x <\infty.
   $$
   From \cite[Cor.\ 2.4]{Kauhanen99} it follows that $Du$
   is of class $L^{n-1,1}$.
 \end{proof}

 The following convention will be used throughout the paper.
\begin{convention}\label{co:convention}
If $\vec u:\Omega \to \R^n$ is measurable 
and $\vec u|_{\partial U}\in W^{1,A}(\partial U, \R^n)$ for some $C^1$ open set $U \ssubset \Omega$ and some 
$N$-function $A$ satisfying \eqref{eq:L2logL} and the $\Delta_2$-condition 
  at infinity,
then in expressions like $\vec u(\partial U)$ or $\vec u|_{\partial U}$ 
we shall be referring to the continuous representative of $\vec u|_{\partial U}$ in $W^{1,p}(\partial U, \R^n)$, which 
exists thanks to Proposition \ref{pr:embedding}.
Moreover, we will usually write $\vec u \in W^{1,A}(\partial U, \R^n)$ 
instead of $\vec u|_{\partial U}\in W^{1,A}(\partial U, \R^n)$.
\end{convention}

 Federer's change of variables formula for surface integrals \cite[Cor.\ 3.2.20]{Federer69}
(see also \cite[Prop.~2.7]{MuSp95} and \cite[Prop.~2.9]{HeMo12}), combined with Lusin's property for 
Sobolev maps with gradients in Lorentz spaces
proved by Kahuanen, Koskela \& Mal\'y \cite[Thm.\ C]{Kauhanen99}, will play an important role in the paper.
We will adopt the following formulation.

\begin{proposition} 
 \label{pr:215}
Let $A$ be an $N$-function satisfying \eqref{eq:L2logL} and the $\Delta_2$-condition 
  at infinity.
Suppose that $U$ is a $C^1$ open subset of $\Omega$, and $\vec u|_{\partial U}\in W^{1,A}(\partial U, \R^n)$.
Assume, further, that 
$\nabla (\vec u|_{\p U})(\vec x) = \nabla \vec u(\vec x) |_{T_{\vec x} \p U}$ 
for $\mathcal H^{n-1}$-a.e.\ $\vec x\in \p U$.
Then, for any $\mathcal H^{n-1}$-measurable subset $E\subset \partial U$,
\begin{equation*}
\mathcal H^{n-1}(\vec u(E)) = \int_E |(\cof \nabla \vec u(\vec x))\vecg\nu(\vec x)|\, \dd \mathcal H^{n-1}(\vec x),
\end{equation*}
where $\vecg\nu(\vec x)$ denotes the outward unit normal to $\partial U$ at $\vec x$.
\end{proposition}

\begin{remark}
 \label{re:MM}
\begin{enumerate}[a)]
\item By $\vec u(E)$ we refer to the image of $E$ by the continuous representative of $\vec u|_{\partial U}$ 
in $W^{1,A}(\partial U, \R^n)$, due to Convention \ref{co:convention}.
\item\label{low dimension image} We are mostly interested in the facts that $\mathcal H^{n-1}(\vec u(\partial U))<\infty$ 
and that $\mathcal H^{n-1}(\vec u(E))=0$ for every $\mathcal H^{n-1}$-null 
set $E\subset \partial U$. In particular, $\mathcal L^n(\vec u(\partial U))=0$, and 
$\vec u(\partial U) = \vec u(\partial U \cap \Omega_0)$ $\mathcal H^{n-1}$-a.e.\ 
if $\partial U \subset \Omega_0$ $\mathcal H^{n-1}$-a.e., where $\Omega_0$ is the set of Definition \ref{de:O0}.
\end{enumerate}
\end{remark}

\subsection{A class of good open sets}\label{subse:good}

In the following definition, given a nonempty open set $U \ssubset \O$ with a $C^2$ boundary, we call $d: \O \to \R$ 
the function given by
\[
 d(\vec x) := \begin{cases} \dist(\vec x, \p U) & \textrm{if}\ \vec x \in U \\
0 & \textrm{if}\ \vec x \in \p U \\
-\dist(\vec x, \p U) & \textrm{if}\ \vec x \in \Omega \setminus \bar U 
\end{cases}
\]
and
\begin{equation}\label{eq:Ut}
 U_t := \left\{ \vec x \in \O : d(\vec x)>t \right\} ,
\end{equation}
for each $t \in \R$.
We note (see, e.g., \cite[Th.\ 16.25.2]{Dieudonne72}, \cite[p.\ 112]{Sverak88} or \cite[p.\ 48]{MuSp95}) that there exists $\d>0$ such that for all $t \in (-\d,\d)$, the set $U_t$ is open, compactly contained 
in $\O$ and has a $C^2$ boundary.

\begin{definition}\label{df:good_open_sets}
Let $A$ be an $N$-function satisfying \eqref{eq:L2logL} and the $\Delta_2$-condition 
  at infinity.
Let $\vec u \in W^{1,A} (\O, \Rn)$.
We define $\mc{U}_{\vec u}$ as the family of nonempty open sets $U \ssubset \O$ with a $C^2$ boundary that satisfy the following conditions:
\begin{enumerate}[a)]
\item $\vec u|_{\p U} \in W^{1,A}(\p U, \Rn)$, and $(\cof \nabla \vec u)|_{\p U} \in L^1 (\p U, \Rnn)$.

\item $\p U \subset \Omega_0$ $\Hn$-a.e., where $\Omega_0$ is the set of Definition \ref{de:O0}, and $\nabla (\vec u|_{\p U})(\vec x) = \nabla \vec u(\vec x) |_{T_{\vec x} \p U}$ for $\mathcal H^{n-1}$-a.e.\ $\vec x\in \p U$.

\item $\displaystyle \lim_{\e \searrow 0} \dashint_0^{\e} \left| \int_{\p U_t} |\cof \nabla \vec u| \, \dd \mc{H}^{n-1} - \int_{\p U} |\cof \nabla \vec u| \, \dd \mc{H}^{n-1} \right| \dd t = 0 .$

\item For every $\vec g \in C^1 (\Rn, \Rn)$ with $(\adj D \vec u) (\vec g \circ \vec u) \in L^1_{\loc} (\O, \Rn)$,
\[
  \lim_{\e \searrow 0} \dashint_0^{\e} \left| \int_{\p U_t} \! \vec g (\vec u (\vec x)) \cdot \left( \cof \nabla \vec u (\vec x) \, \vecg \nu_t (\vec x) \right) \dd \mc{H}^{n-1} (\vec x) - \int_{\p U} \! \vec g (\vec u (\vec x)) \cdot \left( \cof \nabla \vec u (\vec x) \, \vecg \nu (\vec x) \right) \dd \mc{H}^{n-1} (\vec x) \right| \dd t = 0 ,
\]
where $\vecg \nu_t$ denotes the unit outward normal to $U_t$ for each $t \in (0, \e)$, and $\vecg \nu$ 
the unit outward normal to~$U$.
\end{enumerate}
\end{definition}

The following result can be proved as in \cite[Lemma 2.9]{MuSp95}.
It is a consequence of Fubini's theorem and 
the compact embedding of $W^{1,A}$
into the space of continuous functions (see \cite[Corollary 1]{Cianchi96},
which is proved for strongly Lipschitz domains and can be used in our setting,
via local charts, since the mainfolds $\partial U_t$ have no boundary).

\begin{lemma}
  \label{le:224}
  Let $A$ be an $N$-function satisfying \eqref{eq:L2logL} and the $\Delta_2$-condition 
  at infinity.
For each $j \in \N$ let $\vec u_j, \vec u \in W^{1,A} (\O, \Rn)$ satisfy $\vec u_j \weakc \vec u$ in $W^{1,A} (\O, \Rn)$ as $j \to \infty$.
Let $U \ssubset \O$ be an open set with a $C^2$ boundary.
Then there exists $\d>0$ such that for a.e.\ $t \in (-\d,\d)$,
\[
 \vec u_j , \vec u \in W^{1,A} (\p U_t, \Rn) \quad \text{for all } j \in \N
\]
and, for a subsequence (depending on $t$),
\[
 \vec u_j \to \vec u \quad \text{uniformly on } \p U_t \ \text{ as } j \to \infty .
\]
\end{lemma}

\subsection{Degree for Orlicz-Sobolev maps} \label{se:degree}

We assume that the reader has some familiarity with the topological degree for continuous functions (see, e.g.,
\cite{Deimling85,FoGa95book}).
Let $U$ be a bounded open set of $\Rn$ and let $\vecg \phi : \p U \to \Rn$ be continuous.
By Tietze's theorem, it admits a continuous extension $\tilde{\vecg \phi} : \bar{U} \to \Rn$.
We define the degree $\deg (\vecg \phi, U, \cdot) : \Rn \setminus \vecg \phi (\p U) \to \Z$ of $\vecg \phi$ on $U$ 
as the degree $\deg (\tilde{\vecg \phi}, U, \cdot) : \Rn \setminus \vecg \phi (\p U) \to \Z$ of $\tilde{\vecg \phi}$ on $U$.
This definition is consistent since the degree only depends on the boundary values (see, e.g., \cite[Th.\ 3.1 (d6)]{Deimling85}).

The following formula for the distributional derivative of the degree will be widely used
(see, e.g., \cite[Prop.\ 2.1]{MuSpTa96} or \cite[Prop.\ 2.1]{MuSp95}).

\begin{proposition}\label{prop:degdiv}
Let $A$ be an $N$-function satisfying \eqref{eq:L2logL} and the $\Delta_2$-condition 
  at infinity.
Let $U \subset \Rn$ be a $C^1$ open set.
Suppose that $\vec u$ is the continuous representative of a function in $W^{1,A} (\p U, \Rn)$.
Then, for all $\vec g \in C^1 (\Rn, \Rn)$,
\[
 \int_{\p U} \vec g (\vec u (\vec x)) \cdot \left( \cof D \vec u (\vec x) \, \vecg \nu (\vec x) \right) \dd \mc{H}^{n-1} (\vec x) = \int_{\Rn} \div \vec g (\vec y) \deg (\vec u, U, \vec y) \, \dd \vec y ,
\]
where $\vecg \nu$ is the unit outward normal to $U$.
\end{proposition}

\begin{proof}
 As mentioned in \cite[Prop.\ 2.1, Rmk.\ 2]{MuSp95}, for the formula to be valid is 
enough to know that $\vec u\in W^{1,n-1}(\p U, \Rn)$, that $\vec u$ has a continuous
representative and that $\mathcal L^n(\vec u(\p U))=0$.
That $W^{1,A}(\p U, \Rn)\subset W^{1,n-1}(\p U, \Rn)$ follows from
the fact that $L^A(\p U)
\subset L^{n-1}(\p U)$.
Functions in $W^{1,A}(\p U, \Rn)$ satisfy the remaining two conditions thanks 
again to 
Proposition \ref{pr:embedding} and Remark \ref{re:MM}.(\ref{low dimension image}).
\end{proof}

The concept of topological image was introduced by \v{S}ver\'ak \cite{Sverak88} (see also \cite{MuSp95}).

\begin{definition}\label{def:INV_CDL}
Let $A$ be an $N$-function satisfying \eqref{eq:L2logL}
and let $U \ssubset \Rn$ be a nonempty open set with a $C^1$ boundary.
If $\vec u \in W^{1,A}(\p U, \Rn)$, we define $\imT(\vec u, U)$, the topological image of $U$ under 
$\vec u$, as the set of $\vec y \in \Rn \setminus \vec u (\p U)$ such that $\deg(\vec u, U, \vec y) \neq 0$.
\end{definition}

Due to the continuity of $\deg (\vec u, U, \vec y)$ with respect to $\vec y$,
the set $\imT(\vec u, U)$ is open
and $\partial \imT(\vec u, U) \subset \vec u(\partial U)$. In addition, as $\deg(\vec u, U, \cdot)$ is zero 
in the unbounded component of $\Rn \setminus \vec u (\p U)$ (see, e.g., \cite[Sect.\ 5.1]{Deimling85}), 
it follows that $\imT (\vec u, U)$ is bounded.

\subsection{Distributional determinant}
We present the definition of distributional determinant (see \cite{Ball77} or \cite{Muller90CRAS}).
With $\langle \cdot , \cdot \rangle$ we indicate the duality product between a distribution and a smooth function.

\begin{definition}\label{de:Det}
Let $\vec u \in W^{1,1} (\O, \Rn)$ satisfy $(\adj D \vec u) \, \vec u \in L^1_{\loc} (\O, \Rn)$.
The distributional determinant of $\vec u$ is the distribution $\Det D \vec u$ defined as
\[
 \langle \Det D\vec u , \phi \rangle := -\frac{1}{n} \int_\Omega \vec u(\vec x) \cdot (\cof D\vec u(\vec x)) \, D\phi(\vec x) \, \dd \vec x, \qquad \phi\in C_c^\infty(\Omega) .
\]
\end{definition}

\subsection{Surface energy}
 
 The following concepts were defined in \cite{HeMo10}:
 \begin{definition}\label{def:SE}
Let $\vec u : \O \to \Rn$ be measurable and approximately differentiable a.e\@.
Suppose that $\det \nabla \vec u \in L^1_{\loc} (\O)$ and $\cof \nabla \vec u \in L^1_{\loc} (\O,\Rnn)$.
For every $\vec f \in C^1_c (\O \times \Rn,\Rn)$, define
\begin{equation}\label{eq:Euf}
 \mc{E} (\vec u, \vec f) := \int_{\O} \left[ \cof \nabla \vec u (\vec x) \cdot D \vec f (\vec x, \vec u (\vec x)) + \det \nabla \vec u (\vec x) \div \vec f (\vec x, \vec u (\vec x))  \right] \dd \vec x
\end{equation}
and
\[
 \mc{E} (\vec u) := \sup \left\{ \mc{E} (\vec u, \vec f) : \ \vec f \in C^1_c (\O \times \Rn,\Rn), \ \| \vec f \|_{\infty} \leq 1 \right\} .
\]
\end{definition}
In equation \eqref{eq:Euf}, $D \vec f (\vec x, \vec y)$ denotes the derivative of $\vec f (\cdot, \vec y)$ 
evaluated at $\vec x$, while $\div \vec f (\vec x, \vec y)$ is the divergence of $\vec f (\vec x, \cdot)$ 
evaluated at $\vec y$.

It was proved in \cite{HeMo11,HeMo12} that if $\vec u$ is one-to-one a.e., $\det \nabla \vec u>0$ a.e.\ 
and $\mathcal E(\vec u)<\infty$ then 
$$\mathcal E(\vec u)= \mathcal H^{n-1}(\Gamma_V(\vec u)) + 2\mathcal H^{n-1}(\Gamma_I(\vec u)),$$
where $\Gamma_V(\vec u)$ and $\Gamma_I(\vec u)$ are $(n-1)$-rectifiable sets, defined as follows:
\begin{itemize}
 \item A point $\vec y_0$ belongs to $\Gamma_V(\vec u)$ if the approximate limit of $\vec u^{-1}(\vec y)$ as $\vec y$
approches $\vec y_0$ from one side of $\Gamma_V(\vec u)$ lies in the interior of $\Omega$, and either there
are almost no points of $\imG(\vec u, \Omega)$ on the other side of $\Gamma_V(\vec u)$ or the approximate limit
of $\vec u^{-1}(\vec y)$ coming from the other side lies on the boundary of $\Omega$.
 \item A point $\vec y_0$ belongs to $\Gamma_I(\vec u)$ if the approximate limits of $\vec u^{-1}(\vec y)$ coming from
 the two sides of $\Gamma_I(\vec u)$ exist, are different, and both lie in the interior of $\Omega$.
\end{itemize}
The motivation there was the modelling of fracture, context in which $\Gamma_V(\vec u) \cup \Gamma_I(\vec u)$
corresponds to the surface created by the deformation, as seen in the deformed configuration. 
In that case $\mathcal E(\vec u)$ gives the area of this created surface.

\subsection{Weak monotonicity}
  The following definition of weak monotonicity was introduced by Manfredi \cite{Manfredi94}
  (see, e.g., \cite{VoGo76} for earlier related definitions; the subscript $+$ stands for positive part).
  
  \begin{definition}
   \label{de:monotone}
    A function $u\in W^{1,1}_{\loc}(\Omega)$ is called weakly monotone if, 
    for every open set $\Omega' \ssubset \Omega$, and every $m,M\in\R$,
    such that $m\leq M$ and $$(u-M)_+-(m-u)_+\in W^{1,1}_0(\Omega'),$$
    one has that $$m\leq u \leq M\quad \text{a.e.\ in}\ \Omega'.$$
  \end{definition}
  
  The definition asks for a weak version of the minimum and maximum principle
  to be satisfied for every open $\Omega'\ssubset \Omega$. 
  We shall work with maps where that minimum and maximium principles are satisfied
  only for open sets in $\mathcal U_{\vec u}$;
  in particular, given any $\vec x$ in $\Omega$ we will only be able to assume
  that they hold for a.e.\ $r\in (0, \dist (\vec x, \partial \Omega))$ and not for every such radius.
  This possibility was taken into account in the notion of weak pseudomonotonicity 
  of Hajlasz \& Mal\'y \cite{HaMa02}
  (which, in fact, is more general than what we need: we will only consider the case $K=1$).
  
  \begin{definition}
   \label{de:pseudomonotone}
   A map $u\in W^{1,1}(\Omega)$ is said to be weakly $K$-pseudomonotone, $K\geq 1$,
   if for every $\vec x\in \Omega$ and a.e.\ $0<r<\dist(\vec x, \partial \Omega)$,
   $$\essosc_{B(\vec x,r)} u \leq K \essosc_{S(\vec x, r)} u,$$
   where the oscillation on the left is essential with respect to the Lebesgue measure and
   the oscillation on the right is essential with respect to the $(n-1)$-dimensional Hausdorff
   measure.
  \end{definition}

\section{$H^1$-continuity of pseudomonotone Orlicz-Sobolev maps} \label{se:H1continuity}
In the paper \cite{Carozza19} the authors develop continuity properties of weakly 
monotone Orlicz Sobolev functions. In our analysis, 
we improve their estimate concerning the Hausdorff dimension of points where the function is not continuous.
Also, since in the following sections this estimate will be needed for maps
whose restrictions to balls $B(\vec x, r)$ we will only be able to prove 
that satisfy the weak minimum and maximum principles for a.e.\ $r$ (instead of for every $r$),
we show that their arguments remain valid under this milder monotonicity 
condition. 
We take the chance for a slight generalization and obtain the oscillation estimates
assuming only that the maps are pseudomonotone.

Given a continuous, increasing function $h:[0,\infty) \to [0,\infty)$
such that $h(0)=0$, the $h$-Hausdorff measure $\mc{H}^{h(\cdot)}(E)$ of a set $E\subset \Rn$
is defined as 
\begin{align}
 \mc{H}^{h(\cdot)}(E) = \lim_{\delta\searrow 0}
  \inf \Big \{ \sum_{j=1}^\infty h(\diam(K_j)):\ 
  E\subset \bigcup_{j=1}^\infty K_j,\ 
  \diam(K_j)\leq \delta \Big \}.
\end{align}

\begin{lemma}
 \label{le:Hh}
Let $A$ be an $N$-function satisfying \eqref{eq:L2logL}
and the $\Delta_2$-condition at infinity.
Set  
\begin{align}
 \label{eq:h}
 h(r):= \int_0^r t^{n-1} A_{n-1}\left(\frac1t\right) \dd t,
\end{align}
where 
$A_{n-1}$ is the Young function given by
 \begin{align}\label{an-1}
  A_{n-1}(t):=
               \left ( t^{\frac{n-1}{n-2}} \int_t^\infty 
		\frac{\widetilde A(r)}{r^{1+\frac{n-1}{n-2}}}\dd r \right )^{\widetilde{ }}.
 \end{align}
For every $f\in L^1(\Omega)$
\begin{align}
 \mc{H}^{h(\cdot)}( \{ \vec x_0\in \Omega: \limsup_{r\searrow 0} 
  \frac{\int_{\Omega} |f|\dd\vec x}{h(r)} >0 \} ) =0.
\end{align}
\end{lemma}
 
\begin{proof}
 We will follow \cite[Thm.\ 2.4.3.3]{ EvGa92}.
 Let us show first that 
 $$E:=\{ \vec x_0\in \Omega: \limsup_{r\searrow 0} 
  \frac{\int_{B(\vec x_0, r)} |f|\dd\vec x}{h(r)} >0 \} $$
 does not contain any Lebesgue-Haussdorff point of $f$.
 Indeed,
 \begin{align}
  \label{eq:H1Hh}
  \forall r\in (0,1):\ h(r)\geq A_{n-1}(1) r
 \end{align}
 because $t \to t^{n-1} A_{n-1} \left ( \frac{ 1} {t}\right )$ is decreasing
 \cite[Eq.\ (4.16)]{Carozza19}.
 Hence, if $\vec x_0$ is a Lebesgue point of $f$ then 
 \begin{align*}
  \limsup_{r\searrow 0}\frac{\int_{ B(\vec x_0, r)} |f|\dd x }{h(r)}
  \leq \limsup_{r\searrow 0} \dashint_{ B(\vec x_0, r)} |f|\dd x \cdot \frac{Cr^{n}}{A_{n-1}(1)r} 
 = 0.
 \end{align*}
 As a consequence, for all $\sigma>0$ we can find an open set 
 $U \subset \Omega$ such that $U\subset E$ and $\int_{U} |f(x)| \dd x<\sigma$,
 using the absolute continuity of the density $|f(x)|$. 
 Fix $\e>0$, and define
\begin{align}
E^{\e} \colon= \{x_0\in \Omega: \limsup_{r\to 0} \frac{\int_{ B(\vec x_0, r)} |f(x)|\dd x }{h(r)}>\e \}.
\end{align}
We will prove that $\mathcal H^{h(\cdot)} (E^{\epsilon})=0.$ 
By Vitali's covering theorem, 
for any $\delta>0$ there exist disjoint balls $(B_i)_{i\in \mathbb{N}}$ 
such that $E^{\e}\subset \bigcup_{i\in \mathbb{N}} 5B_i$, $B_i \subset U$, $r_i=\diam(B_i)<\delta$, 
$\int_{B_i}|f| \dd x>\e h(r_i)$.
Using that $A_{n-1}$ is increasing and the definition of $h(r)$ it is straightforward to show that $h(5r)\le 5^n h(r),\forall r>0.$
We then proceed in the estimate:
\begin{equation}
\mathcal H^{h(\cdot)} (E^{\e}) \leq \sum_{i=1}^{\infty} h(5 r_i )
\leq 5^n \sum_{i=1}^{\infty} \frac{\int_{B_i } |f| \dd x}{\e}<\frac{5^n \sigma}{\e}.
\end{equation}
The conclusion follows by letting $\delta \to 0$ and then $\sigma \to 0$. Since $\mathcal H^{h(\cdot)} (E^{\epsilon})=0, \forall \epsilon >0$, we conclude that $\mathcal H^{h(\cdot)} (E)=0.$
\end{proof}

We remark that the weak minimum and maximum principle holds a.e.\  (see Prop. 5.5 in \cite{BHM17}).
We would like to apply the estimate as in \cite[Lemma 7.4.1]{IwMa01}
in order to obtain the following 
 Orlicz version of Gehring oscillation estimate (\cite[Lemma 7.4.2]{IwMa01}):
\begin{equation} 
\label{eq:osc}
  \begin{gathered}
  \textit{If $\vec a$ and $\vec b$ are  Lebesgue points of $f$ 
  and $B(\vec x_0, r)\ssubset \Omega$ is any ball containing $\vec a$ and $\vec b$,}\\
  \textit{then, for a.e.\ }t\in (r,\dist(\vec x_0, \partial \Omega))\textit{,}
  \ 
t^{n-1}A_{n-1}\Big(\frac{|f(\vec a)-f(\vec b)|}{CKt}\Big) \le  \int_{S_t(\vec x_0)} A(|\nabla f|) d \mathcal{H}^{n-1}.
  \end{gathered}
\end{equation}

\begin{proposition}
Let $A$ be a Young function that fulfills condition \eqref {eq:L2logL} for $n\ge 3$. 
Let $A_{n-1}$ be the function defined in \eqref{an-1}.
If $f\in W^{1,A}_{loc}(\Omega)$ and is $K$-pseudomonotone
then \eqref{eq:osc} holds.
\end{proposition}

\begin{proof} The proof simplifies  the one presented in \cite[Thm.\ 3.1]{Carozza19}. 

Let $\vec a$ and $\vec b$ be Lebesgue points of $f$ in $B_r(\vec x_0)$.
Since
\begin{align}
 |f(\vec a)-f(\vec b)| = \left | \lim_{\rho \to 0} \dashint_{B(\vec 0, \rho)}
 (f(\vec a +\vec z ) - f(\vec b + \vec z)) \dd\vec z \right |
 \leq \limsup_{\rho\to 0} \dashint_{B(\vec 0, \rho)} |f(\vec a +\vec z ) - f(\vec b + \vec z)|\dd\vec z; 
\end{align}
for almost every $\tau\in (r,R)$
\begin{align}
  \essosc_{B(\vec x_0, \tau)} f \leq K\essosc_{S(\vec x_0, \tau)} f;
\end{align}
and for every $\rho < \min\{r-|\vec a -\vec x_0|, r-|\vec b - \vec x_0|\}$
and a.e.\ $\tau\in (r,R)$
\begin{align}
  |f(\vec a +\vec z ) - f(\vec b + \vec z)|\leq \essosc_{B(\vec x_0, \tau)} f \quad \text{for a.e.\ }z\in B(\vec 0, \rho);
\end{align}
it follows that 
\begin{equation}\label{pseudomon}
|f(\vec a)-f(\vec b)|\leq K \essosc_{S_\tau(\vec x_0)} f\quad \text{for a.e.\ }\tau\in (r,R).
\end{equation}
At this point, for a.e.\ $\tau>0$ the Poincar\'e-Sobolev inequality, \cite[Thm.\ 4.1]{Carozza16}, 
on the $(n-1)$-dimensional sphere for functions in $W^{1,A}(B_\tau)$ holds:
\begin{equation}
\essosc_{S_\tau} f\le C \tau A_{n-1}^{-1}\Big( \tau^{1-n} \int_{S_\tau } A(|\nabla f|) d \mathcal{H}^{n-1}\Big).
\end{equation}
The proof is finished by combining  \eqref{pseudomon} with the Poincar\'e-Sobolev
inequality.
\end{proof}

One part of the proof of \cite[Thm. 3.1]{Carozza19}
consists in obtaining the estimate \eqref{eq:osc2} below and the a.e.\ differentiability of Orlicz maps 
from
the Gehring oscillation estimate \eqref{eq:osc} (stated in \cite{Carozza19} as Eq.\ (4.15)).
In order to make this connection more explicit we state it as a separate proposition.

\begin{proposition}
If $f\in W^{1,A}_{loc}(\Omega)$ and satisfies \eqref{eq:osc} then
$f\in L^{\infty}_{loc}(\Omega)$ and there exists a constant $c=c(n)$ such that
\begin{equation} \label{eq:osc2}
\essosc _{B_r} f\le c K r A_{n-1}^{-1}\Big(\dashint_{B_{2r}} A(|\nabla f|) d x \Big)
\end{equation}
whenever $B_{2r}\ssubset \Omega$. Moreover, there exists a representative of $f$ that is differentiable a.e.
\end{proposition}

\begin{remark}
As explained in \cite[Rmk.\ 3.2]{Carozza19}, 
another way of seeing 
that weakly monotone maps with $\int A(|\nabla f|)\dd\vec x<\infty$ for some $A$
satisfying \eqref{eq:growth2}
are a.e.\ differentiable is by recalling that maps with this integrability 
have gradients in the Lorentz space $L^{n-1,1}$ (thanks to \cite{Kauhanen99}, see Proposition \ref{pr:embedding} above)
and that weakly monotone maps with $\nabla f \in L^{n-1,1}$ 
were proved to be a.e.\ differentiable in \cite[Thm.\ 1.2]{Onninen00}.
\end{remark}

\begin{proposition}
  \label{pr:H1continuity}
 Let $A$ be an $N$-function satisfying \eqref{eq:L2logL}
and the $\Delta_2$-condition at infinity.
For every K-pseudomonotone map $u$ in $W^{1,A}(\Omega)$
\begin{align}
 \mathcal H^{h(\cdot)}(\{\vec x_0\in \Omega: \limsup_{r\searrow 0}
  \essosc_{B(\vec x_0, r)} u >0\})=0.
\end{align}
\end{proposition}

\begin{proof} 
 Using \eqref{eq:osc} as in the proof of \cite[Thm.\ 3.3]{Carozza19}
 it can be seen that 
 given any $\vec x_0\in \Omega$, 
 and any $r>0$ such that $B(\vec x_0, r)\ssubset \Omega$
 \begin{align}
  \label{eq:oscillation}
  t^{n-1} A_{n-1} \left ( \frac{\essosc_{B(\vec x_0,r)} u}{CKt}\right ) 
  \leq \int_{S(\vec x_0, t)} A(|D u|)\dd\mathcal H^{n-1}
 \end{align}
 for a.e.\ $t\in (r, \dist(\vec x_0, \partial \Omega))$.
 Using \eqref{eq:oscillation} instead of the classical 
 oscillation estimate for weakly monotone Sobolev maps,
 we proceed as in \cite[Thm.\ 7.4]{MuSp95}.
 Set 
 \begin{align}
  E:=\{\vec x_0\in \Omega: \limsup_{r\searrow 0}
  \essosc_{B(\vec x_0, r)} u >0\}
 \end{align}
 and let $\vec x_0\in E$.
 Then there exists $\lambda>0$ such that for a.e.\ $t <\dist(\vec x_0, \partial \Omega)$
\begin{align}
\int_{S(\vec x_0, t)} A(|D u|)\dd x \ge t^{n-1} A_{n-1} \left ( \frac{ \lambda} {CKt}\right ).
\end{align}
By \cite[Prop.\ 4.3]{Carozza19}, $A_{n-1}$ satisfies the $\Delta_2$ condition at infinity.
Hence, $$A_{n-1}\left ( \frac{ \lambda} {CKt}\right ) \geq (C')^{-(1+\log_2(CK/\lambda))} 
 A_{n-1} \left ( \frac{ 1} {t}\right )\quad \forall t < \frac{\lambda}{CKt_0},$$
 for some fixed positive $t_0$ and $C'$. 
Integrating over the interval $[0, r]$:
\begin{align}
\limsup_{r\searrow 0}\frac{\int_{B(\vec x_0, r)} A(|D u|)}{h(r)}\dd x\ge (C')^{-(1+\log_2(CK/\lambda))},
\end{align}
with $h$ defined as in \eqref{eq:h}.
The result then follows by applying Lemma \ref{le:Hh} to $f(x):=A(|D u(x)|)$.
\end{proof}

\begin{remark}
It follows from \eqref{eq:H1Hh} that
\begin{align}  
 \label{eq:H1Hh2}
 \mathcal H^1(E) \leq \frac{2}{A_{n-1}(1)} \mathcal H^{h(\cdot)}(E)
\end{align}
for every Borel set $E\subset \Rn$.
This will allow us to define, in Section \ref{se:classA}, a precise representative
of $u$ that is continuous outside an $\mc{H}^1$-null set.
This improves the result that 
$u$ is $\mathcal H^{h(\cdot)}$-continuous 
with $h(s)= s\log^{-\gamma}(\frac1s)$, for all $\gamma>n-2-\alpha$, in \cite[Example \ 5.1(iii)]{Carozza19}.
More generally, 
neither Proposition \ref{pr:H1continuity}
nor the $\mc{H}^1$-continuity are a consequence 
of \cite[Thm.\ 3.6]{Carozza19}.
Indeed, in order to obtain the $\mathcal H^1$-continuity from \cite[Thm.\ 3.6]{Carozza19}
we would need that
\begin{align} 
 \label{eq:int36}
\int_0 h(s) d\left( \frac{-1}{s^n \sigma(\frac1s) A_{n-1}(\frac 1s)}\right) \dd s<\infty
\end{align}
for $h(s)=s$
and some continuous function $\sigma:[0,\infty)\to[0,\infty)$ 
such that $\int^\infty \frac{\dd t}{t \sigma(t)}=\infty$,
but it can be shown that for any such $\sigma$ the integral 
in \eqref{eq:int36} is not convergent near $0$.
\end{remark}

\section{Orientation-preserving functions creating no new surface}
 \label{se:classA}

Our analysis is set up in the following functional class, 
for a given $N$-function $A$ satisfying \eqref{eq:L2logL}
and the $\Delta_2$-condition at infinity.

\begin{definition} 
  \label{df:classA}
 We define $\mathcal A$ as the set of 
 $\vec u \in W^{1,A}(\Omega, \R^n)$ such that $\det D\vec u \in L^1_{\loc}(\Omega)$,
 $\det D\vec u>0$ a.e.\ and $\mathcal E(\vec u)=0$.
\end{definition}

Intuitively, the maps that satisfy $\det D\vec u>0$ a.e.\ and $\mathcal E(\vec u)=0$ are those 
for which $\partial \vec u(\Omega)=\vec u(\partial \Omega)$ (recall the interpretation of $\mathcal E(\vec u)$
as the area of the surface created by $\vec u$, mentioned after Definition \ref{def:SE}).
It can be seen, using the density of the linear combinations of functions of separated variables,
that $\mathcal E(\vec u)=0$ if and only if 
$$\Div \big ( (\adj D\vec u) \vec g \circ \vec u \big ) =((\div g)\circ \vec u )\det D\vec u\quad
 \forall \vec g \in C_c^\infty (\Rn, \Rn).$$
 This is a regularity requirement. The identity is satisfied by $C^2$ maps $\vec u$, thanks to 
 Piola's identity. It is closely related to the well-known equation $\Det D\vec u=\det D\vec u$,
 satisfied by all $W^{1,n}$ maps. In fact, for maps in $W^{1,p}$ with $p>n-1$
 it was proved in \cite[Corollary 4.7]{BHM17} that $\det D\vec u>0$ a.e.\ and $\mathcal E(\vec u)=0$
 if and only if $\det D\vec u(\vec x)\ne 0$ for a.e.\ $\vec x\in \Omega$, 
 $\Det D\vec u=\det D\vec u$, and $\deg (\vec u, B, \cdot)\geq 0$
 for every ball $B$ belonging to $\mathcal U_{\vec u}$. 
 The condition $\deg (\vec u, B, \cdot)\geq 0$ for all $B$ is known in topology to be the 
 right way to express that $\vec u$ preserves orientation. Along these lines 
 it was proved in \cite[Thm.\ 7.2]{HeMo12} that without the regularity requirement 
 that $\mathcal E(\vec u)=0$ the condition $\det D\vec u>0$ a.e.\ is insufficient to ensure
 the preservation of orientation and the positivity of the Brouwer degree, even 
 if $\Det D\vec u=\det D\vec u$.

\subsection{Fine properties}

Recall the notation $\mathcal N$ from Section \ref{subse:approximate}.

\begin{proposition}
 \label{th:char1}
Every $\vec u \in \mathcal A$ satisfies:
  \begin{enumerate}[a)]
   \item $\vec u \in L^\infty_{\loc}(\Omega, \Rn)$.
   \item $\Det D \vec u = \det D \vec u$.
   \item For all $U \in \mc{U}_{\vec u}$,
 \begin{equation}\label{eq:degUNU}
 \deg (\vec u, U, \cdot) = \mc{N}_U \ \text{ a.e.}
 \quad \text{and}\quad \imT(\vec u, U)=\imG(\vec u, U)\ \text{a.e.}
  \end{equation}
   \item For every $U_1$, $U_2\in \mc{U}_{\vec u}$ with $U_1\ssubset U_2$,
   \begin{equation}
    \label{pr:43d}
 \deg (\vec u, U_1, \cdot) \leq \deg (\vec u, U_2, \cdot) \text{ a.e.\ and in } \Rn \setminus \vec u (\p U_1 \cup \p U_2) ,
\quad\text{and}\quad \overline{\imT(\vec u, U_1)} \subset \overline{\imT(\vec u, U_2)}.
\end{equation}
   \item The components of $\vec u$ are weakly 1-pseudomonotone.
  \end{enumerate}
\end{proposition}

\begin{proof}
 The equalities in \eqref{eq:degUNU}, that $D\vec u\in L^\infty_{\loc}(\Omega,\Rn)$
 and that $\Det D\vec u=\det D\vec u$ can be proved exactly 
 as in \cite[Thm.\ 4.1]{BHM17}. The monotonicity of the degree
 follows with the same proof of \cite[Prop.\ 4.3.(d)]{BHM17},
 taking into account that $\deg (\vec u, U, \cdot)\geq 0$ in $\R^n \setminus \vec u(\p U)$
 by virtue of \eqref{eq:degUNU}.
 Finally, the weak 1-pseudomonotonicity can be established exactly as in \cite[Prop.\ 5.5]{BHM17}.
\end{proof}

\begin{remark}
 The statement in \cite[p.\ 773]{BHM17} that the conditions that $\Det D\vec u=\det D\vec u$
 and $\det D\vec u>0$ a.e.\ are enough to ensure that the components of $\vec u$ are
 weakly monotone is incorrect. The construction in \cite[Thm.\ 7.2]{HeMo12} constitutes a counterexample.
 We were not able to determine whether the stronger condition that $\mathcal E(\vec u)=0$
 renders the conclusion true. 
\end{remark}

It is well known (see, e.g., \cite[Ch.\ 2]{HeKo14})
that the weak monotonicity implies regularity properties. In particular, for $W^{1,p}$-maps
with $p>n-1$, a representative of $\vec u$ is continuous $\mc{H}^{n-p}$-a.e.\ 
(if $p \leq n$) and differentiable a.e\@. In our case, we get that $\vec u$ is continuous $\mc{H}^{h(\cdot)}$-a.e., where $h$ is defined in \eqref{eq:h}.
However, we will not deal with the representative normally used in the theory of monotone maps 
(see, e.g., \cite{Sverak88,Manfredi94,Vodopyanov00,HaMa02,HeKo14}) 
but rather with the one defined in \cite[Th.\ 7.4]{MuSp95}, which we explain in the following paragraphs.

\begin{definition}\label{de:imTux}
Let $\vec u \in \mc{A}$.
We define the \emph{topological image} of a point $\vec x \in \Omega$ by $\vec u$ as 
\begin{equation*}
 \imT(\vec u, \vec x) := \bigcap_{\substack{r>0 \\ B(\vec x, r) \in \mathcal U_{\vec u}}} \overline{\imT(\vec u, B(\vec x,r))},
\end{equation*}
and $NC:=\{\vec x\in\Omega \colon \mc{H}^0 (\imT(\vec u,\vec x)) >1\}$.
\end{definition}

As explained in \cite[Rmk.\ 5.7.(c)]{BHM17}, neither the topological image of a point nor the set $NC$
depend on the particular representative of $\vec u$ 
(if $\vec u_1, \vec u_2\in \mathcal A$ and $\vec u_1=\vec u_2$ a.e.\ then $\imT(\vec u_1, \vec x)
=\imT(\vec u_2, \vec x)$ for every $\vec x \in \Omega$ and the set $NC$ defined through $\vec u_1$
coincides with the one defined through $\vec u_2$).

\begin{proposition} 
 \label{pr:fine}
For every $\vec u \in \mathcal{A}$ the following are satisfied:
\begin{enumerate}[a)] 
 \item $\mc{H}^1(NC)=0$.
 \item For every $\vec x_0 \in \Omega \setminus NC$ the function 
 $\displaystyle r\mapsto \dashint_{B(\vec x_0, r)} \vec u(\vec x)\,\dd \vec x$ 
 converges, as $r\searrow 0$, to some $\vec u^*(\vec x_0)\in \Rn$.  
 \item The map $\hat{\vec u}$ defined \emph{everywhere} in $\Omega$ by
 \begin{equation}
 \hat {\vec u}(\vec x):=
\begin{cases} 
  \vec u^*(\vec x) & \text{if }
  \vec x \in \Omega \setminus NC, \\ 
\text{any element of $\imT(\vec u,\vec x)$} & \text{if } \vec x \in NC  
\end{cases}
\end{equation}
 is such that $\hat{\vec u}(\vec x)=\vec u(\vec x)$ for every $\vec x \in \Omega_0$ and
 $\hat {\vec u}(\vec x)\in \imT(\vec u, \vec x)$ for every $\vec x\in \Omega$.
 Moreover, it is continuous at every point of $\vec x \in \Omega\setminus NC$,
 differentiable a.e., and such that $\mc{L}^n(\hat {\vec u}(N))=0$
 for every $N\subset \Omega$ with $\mc{L}^n(N)=0$.
\end{enumerate}

\end{proposition}

\begin{proof}
 Let $\vec u \in \mathcal{A}$.
 Denote by $P$ the set of points $\vec x_0\in \Omega$ where the 
 following property fails: 
 there exists $\vec u^*(\vec x_0)\in \Rn$
 such that 
 $$\lim_{r\searrow 0} \dashint_{B(\vec x_0, r)}
 |\vec u(\vec x) - \vec u^*(\vec x_0)|^{n(n-1)}\,\dd\vec x =0.$$
 Since $\mathcal A\subset W^{1,n-1}(\Omega,\Rn)$,
 $P$ has zero $(n-1)$-capacity 
 (see \cite{Ziemer89,EvGa92} or, e.g., \cite[Prop.\ 2.8]{MuSp95}).
 Define
 \begin{equation*}
 \tilde {\vec u}(\vec x):=
\begin{cases} 
\vec u^*(\vec x) & \text{if } \vec x \in \Omega \setminus (P\cup NC), \\ 
\text{any element of $\imT(\vec u,\vec x)$} & \text{if } \vec x \in P\cup NC 
\end{cases}
\end{equation*}
 (use is being made of the axiom of choice).
 
 Let us prove that $\vec u^*(\vec x_0)\in \imT(\vec u, \vec x_0)$ 
 for every $\vec x_0 \in \Omega \setminus (P\cup NC)$.
 Suppose, for a contradiction,
 that $\vec u^*(\vec x_0)\in \Rn \setminus \overline{\imT(\vec u^*, B(\vec x_0, r))}$
 for some $r>0$ such that $B(\vec x_0,r)\in \mathcal U_{\vec u^*}$.
 Since $\deg(\vec u^*, B(\vec x_0, r), \vec y)=0$
 for every $\vec y$ in the open set $\Rn \setminus \overline{\imT(\vec u^*, B(\vec x_0, r))}$,
 the set of points $\vec x \in \Omega$ for which $\deg (\vec u^*, B(\vec x_0, r), \vec u(\vec x))=0$
 would have density $1$ at $\vec x_0$. However, this is incompatible with \eqref{eq:degUNU}.
 
 Proceeding as in Part (b) of the proof of \cite[Thm.\ 7.4]{MuSp95},
 it can be seen that $\tilde{\vec u}$ is continuous 
 at every point of $\vec x \in \Omega\setminus NC$
 (using \eqref{pr:43d} instead of \cite[Lemma 7.3(i)]{MuSp95}).
 One of the consequences of this continuuity is that $P$ is contained in $NC$,
 and, hence, $\tilde{\vec u}(\vec x)=\hat{\vec u}(\vec x)$ for every $\vec x \in \Omega$. 
 
 That $\hat{\vec u}$ satisfies Lusin's property can be proved as in \cite[Th.\ 10.1]{MuSp95}
 (with a slightly shorter proof since $\Det D\vec u=\det D\vec u$).
 
 That $NC$ is an $\mc{H}^1$-null set will be proved at the end.
 At this point, let us show how to obtain the a.e.\ differentiability of $\hat{\vec u}$
 under the assumption that $\mathcal{L}^n(NC)=0$.
 Let $\vec x_1\in \Omega \setminus NC$ be a Lebesgue point for $A(|D\vec u|)$
 and let $\vec x_2\in \Omega\setminus NC$ satisfy $B(\vec x_1, 2(|\vec x_2-\vec x_1|+\rho))\subset \Omega$
 for some $\rho>0$. 
 Let $A_{n-1}$ be the Young function given by
 \begin{align}
  A_{n-1}(t):=
               \left ( t^{\frac{n-1}{n-2}} \int_t^\infty 
		\frac{\widetilde A(r)}{r^{1+\frac{n-1}{n-2}}}\dd r \right )^{\widetilde{ }}.
 \end{align}
 Using \eqref{eq:osc2}
 (with radius $|\vec x_2-\vec x_1|+\rho$) we find that 
 for every $r\in (0,\rho)$ and a.e.\ $\vec z\in B(\vec 0,1)$
 $$
  |\vec u(\vec x_2+r\vec z) - \vec u(\vec x_1+r\vec z)|
  \leq C(|\vec x_1-\vec x_2|+\rho) A_{n-1}^{-1} \left ( \dashint_{B(\vec x_1, 2(|\vec x_1-\vec x_2|+\rho))} 
  A(|D\vec u|)\dd\vec x \right ).
 $$
 Since $\hat{\vec u}$ is continuous outside $NC$,
 \begin{align*}
  |\hat{\vec u}(\vec x_2)-\hat{\vec u}(\vec x_1)|
  &= \left | \lim_{r\searrow 0} \left (\dashint_{B(\vec x_2, r)} \vec u \dd\vec x 
  - \dashint_{B(\vec x_1, r)} \vec u \dd\vec x \right ) 
  \right |
  \\
  &\leq \liminf_{r\searrow 0} \dashint_{B(\vec 0, 1)}
  |\vec u(\vec x_2+r\vec z) - \vec u(\vec x_1+r\vec z)|\dd\vec z 
  \\
  &\leq C(|\vec x_1-\vec x_2|+\rho) A_{n-1}^{-1} \left ( \dashint_{B(\vec x_1, 2(|\vec x_1-\vec x_2|+\rho))} 
  A(|D\vec u|)\dd\vec x \right ).
 \end{align*}
 Letting $\rho\searrow 0$ we find that 
 \begin{equation}\label{eq:Stepanov1}
 \limsup_{\substack{\vec x_2 \to \vec x_1 \\ \vec x_2 \notin NC}} 
 \frac{|\hat{\vec u} (\vec x_2) - \hat{\vec u} (\vec x_1)|}{|\vec x_2 - \vec x_1|} < \infty .
 \end{equation}
 From this point onwards the a.e.\ differentiability can be obtained exactly as in 
 the proof of \cite[Prop.\ 5.9]{BHM17}.

 We now show how to adapt Part (c) of the proof of \cite[Thm.\ 7.4]{MuSp95} in order to obtain
 that $\mc{H}^1(NC)=0$. Set 
 $$E:=\bigcup_{i=1}^n \{\vec x_0 \in \Omega: 
 \liminf_{r\searrow 0} \essosc_{B(\vec x_0, r)} u^i >0\},$$
 where
 $u^i$ denotes the $i$-th component of $\hat{\vec u}$.
 By \eqref{eq:H1Hh2} and Proposition \ref{pr:H1continuity},
 it suffices to show that $NC\subset E$. 
 With this aim observe that
 for every $\vec x_0$ in $NC$ there exists $\lambda>0$ such that
 $\diam \overline{\imT(\hat{\vec u}, B(\vec x_0, r))}>\lambda$
 whenever $B(\vec x_0, r)\in \mathcal U_{\hat{\vec u}}$,
 because $\imT(\vec u, \vec x)$ is contained in 
 $\overline{\imT(\hat{\vec u}, B(\vec x_0, r))}$. 
 By Definition \ref{df:good_open_sets} and Convention \ref{co:convention}, 
 the restriction $\hat{\vec u}|_{\partial B(\vec x_0, r)}$
 may be assumed to be continuous.
 Since $\overline{\imT(\hat{\vec u}, B(\vec x_0, r))}$ is a compact set whose boundary
 is contained in $\hat{\vec u}(\partial B(\vec x_0, r))$,
 there exist $\vec x_1$ and $\vec x_2$ on $\partial B(\vec x_0, r)$
 such that $|\hat{\vec u}(\vec x_2) -\hat{\vec u}(\vec x_1)|>\lambda$. 
 By Definition \ref{df:good_open_sets}, almost every point of $\partial B(\vec x_0, r)$
 belongs to $\Omega_0$. Since $\hat{\vec u}|_{\partial B(\vec x_0, r)}$
 is continuous, without loss of generality we may assume that $\vec x_1$ and $\vec x_2$
 belong to $\Omega_0$. 
 By Definitions \ref{de:O0} and \ref{de:wapproximate}, 
 points in $\Omega_0$
 are points of approximate continuity for $\hat{\vec u}$. 
 As a consequence, there exist measurable sets $A_1,A_2\subset B(\vec x_0,r)$
 of density $\frac{1}{2}$ with respect to $\vec x_1$ and $\vec x_2$, respectively,
 such that 
 $$\forall \vec x_1'\in A_1\ \forall \vec x_2'\in A_2:\ 
 |\hat{\vec u}(\vec x_2')-\hat{\vec u}(\vec x_1')|>\lambda.$$
 Consequently, $$\sum_{i=1}^n \essosc_{B(\vec x_0,r)} u^i > \lambda.$$
 Since this is true for every $r$ such that $B(\vec x_0, r)\in \mathcal U_{\hat{\vec u}}$,
 we conclude that $\vec x_0\in E$, completing the proof.
\end{proof}

\subsection{Openness and properness}

We begin by noting that  equality \eqref{eq:degUNU} implies an openness property for $\vec u$:
		for every $U\in \mc{U}_{\vec u}$,
		\begin{equation}\label{eq:imTUimGU}
		 \imT (\vec u, U) = \imG (\vec u, U) \quad \text{a.e.}
		\end{equation}
		
\begin{definition} \label{de:imTomega}
Let $\vec u\in \mathcal A$, where $\mathcal A$ is that of Definition \ref{df:classA}.
Define $$\mc{U}^N_{\vec u} := \left\{ U \in \mc{U}_{\vec u} \colon \p U \cap NC = \varnothing \right\}$$
and
\[
 \imT (\vec u, \O) := \bigcup_{U\in\mathcal U^N_{\vec u}} \imT (\vec u, U).
\]
\end{definition}

We will see in Section \ref{se:lower} that $\imT (\vec u, \O)$ plays the role of the \emph{deformed configuration}.
By the continuity of the degree, $\imT (\vec u, U)$ is open, and hence, so is $\imT (\vec u, \O)$.
Also, it does not depend on the particular representative of $\vec u$ (the proof of
\cite[Lemma 5.18.(b)]{BHM17} remains valid in our setting).

\begin{proposition} \label{pr:good_Ut}
Let $\vec u \in \mc{A}$.
\begin{enumerate}[a)]
 \item 
  \label{le:UNu}
 For every non-empty open set $U \ssubset \O$ with a $C^2$ boundary 
 there exists $\d>0$ such that $U_t \in \mc{U}^N_{\vec u}$ for a.e.\ $t \in (-\d, \d)$, where $U_t$ is defined as in \eqref{eq:Ut}.
Moreover, for each compact $K \subset \O$ there exists $U' \in \mc{U}^N_{\vec u}$ such that $K \subset U'$.
 \item
  \label{it:KUt}
 For each $U \in \mc{U}^N_{\vec u}$ and 
		each compact $K\subset   \imT (\vec u, U)$
 there exists $\d >0$ such that
\[
		 K\subset
 \bigcap_{\substack{t \in (0,\d) \\ U_t \in \mc{U}^N_{\vec u}}} \imT (\vec u, U_t) .
\]
\end{enumerate}
\end{proposition}

\begin{proof}
 \emph{Part \ref{le:UNu}):} 
 Since, by Proposition \ref{pr:fine}, the set $NC$ is $\mc{H}^1$-null, for each $\vec x\in \Omega$
there exists an $\mathcal L^1$-null set $N\subset (0,\infty)$
such that $NC\cap \partial B(\vec x,r)=\vacio$
for all $r\in (0,\dist(\vec x,\partial \Omega))\setminus N$.
Combining this with \cite[Prop.\ 2.8]{MuSp95} and \cite[Lemma 2 and Def.\ 11]{HeMo11} (or \cite[Lemma 2.16]{HeMo12})
we obtain that there are enough sets in $\mc{U}^N_{\vec u}$ whose boundaries do not intersect $NC$, as claimed.
 \medskip
 
 \emph{Part \ref{it:KUt}):}
 By Convention \ref{co:convention} and Proposition \ref{pr:fine} we have that $\imT(\vec u, U)=\imT(\hat{\vec u}, U)$
 for every $U\in \mathcal U_{\vec u}^N$. Using this and the continuity of $\hat{\vec u}$ at every point of $\partial U$
 the result follows with the same proof of \cite[Lemma 5.18.(a)]{BHM17}.
\end{proof}

\subsection{Local invertibility}

\begin{definition} \label{de:Oinv}
Let $\vec u\in \mathcal A$.
We denote by $\mc{U}^{\inv}_{\vec u}$ the class of $U\in \mathcal U_{\vec u}$ such that $\vec u$ is one-to-one a.e.\ in $U$
(see Definition \ref{df:1-1ae}), 
and by $\mc{U}^{N,\inv}_{\vec u}$ the set $\mc{U}^N_{\vec u} \cap \mc{U}^{\inv}_{\vec u}$. 
Define $$\Omega_{\inv}:=\bigcup \mc{U}^{\inv}_{\vec u}.$$
\end{definition}

The set $\Omega_{\inv}$
consists of the sets of points around which $\vec u$ is locally a.e.\ invertible: 
$\vec x \in \O_{\inv}$ if and only if there exists $r>0$ such that $\vec u$ is one-to-one a.e.\ in $B (\vec x, r)$.
It does not depend on the particular representative of $\vec u$ (as explained after Def.\ 4.4 in \cite{BHM17}).

The local invertibility theorem of Fonseca \& Gangbo \cite{FoGa95paper} for $W^{1,p}$ maps with $p>n$
was generalized, under the assumption  $\mathcal E(\vec u)=0$, to all $p>n-1$. Here it is shown
to hold also in the Orlicz-Sobolev case under the growth condition \eqref{eq:L2logL}.

\begin{proposition}
 For every $\vec u\in \mathcal A$ the set $\O_{\inv}$ is of full measure in $\Omega$.
\end{proposition}

\begin{proof}
 It can be proved that every $\vec x_0\in \Omega$ where $\hat{\vec u}$ is differentiable and $\det D\hat{\vec u}(\vec x_0)>0$
 belongs to $\O_{\inv}$, with the same arguments as in \cite[Proposition 4.5.(d)]{BHM17}.
\end{proof}

Equality \eqref{eq:imTUimGU} makes it possible to define the local inverse having for domain an open set.
\begin{definition}\label{de:inverse}
Let $\vec u \in \mc{A}$ and $U \in \mc{U}^{\inv}_{\vec u}$.
The inverse $(\vec u|_U)^{-1} : \imT (\vec u, U) \to \Rn$ is defined a.e.\ as $(\vec u|_U)^{-1} (\vec y) = \vec x$, 
for each $\vec y \in \imG (\vec u, U)$, and where $\vec x \in U \cap \O_0$ satisfies $\vec u (\vec x) = \vec y$.
\end{definition}

A careful inspection of the proofs shows that 
\cite[Th.\ 3.3]{HeMo15} remains valid in the class $\mathcal A$
or Orlicz-Sobolev maps with positive Jacobian, zero surface energy
and an integrability above $W^{1,n-1}$. (Use is made in
\cite{HeMo15} of the stronger invertibility condition INV of M\"uller \& Spector;
this condition holds for every $U \in \mc{U}^{\inv}_{\vec u}$
thanks to \eqref{eq:imTUimGU}.)

\begin{proposition}\label{prop:locINV}
Let $\vec u \in \mc{A}$ and $U \in \mc{U}^{\inv}_{\vec u}$. Then 
\[
(\vec u|_U)^{-1} \in W^{1,1} (\imT (\vec u, U), \Rn) \quad \text{and} 
\quad D (\vec u|_U)^{-1} = \left( D \vec u \circ (\vec u|_U)^{-1} \right)^{-1} \text{ a.e.}
\]
\end{proposition}

\begin{proposition}\label{th:ujINV}
For each $j \in \N$, let $\vec u_j, \vec u \in \mc{A}$ satisfy $\vec u_j \weakc \vec u$ in $W^{1,A} (\O, \Rn)$ as $j \to \infty$.
The following assertions hold:
\begin{enumerate}[a)]
\item\label{item:thINV2} For any $U \in \mc{U}^N_{\vec u}$ and any 
compact set $K \subset \imT (\vec u, U)$ there exists a subsequence for which 
	${K} \subset \imT (\vec u_j, \O)$ for all $j \in \N$.
\item\label{item:thINV5} For a subsequence, there exists a disjoint family
\[
 \{ B_k \}_{k \in \N} \subset \mc{U}^{N,\inv}_{\vec u} \cap \bigcap_{j \in \N} \mc{U}^{N,\inv}_{\vec u_j}
\]
such that $\O = \bigcup_{k \in \N} B_k$ a.e.\ and, for each $k \in \N$,
\begin{equation}\label{eq:Bkinv}
 \vec u_j \to \vec u\quad  \text{uniformly on } \p B_k, \text{ as } j \to \infty .
\end{equation}


\item\label{item:thINV3} Let $B \in \mc{U}^{\inv}_{\vec u} \cap \bigcap_{j \in \N} \mc{U}^{\inv}_{\vec u_j}$ and take an open set $V \ssubset \imT (\vec u, B)$ such that $V \subset \imT (\vec u_j, B)$ for all $j \in \N$.
Then
\begin{enumerate}[1)]
\item\label{item:ujINVi} $(\vec u_j|_B)^{-1} \weakcs (\vec u|_B)^{-1}$ in $BV (V, \Rn)$ as $j \to \infty$;

\item\label{item:ujINViii} for any minor $M$, we have $M (D (\vec u_j|_B)^{-1})$, $M (D (\vec u|_B)^{-1})\in L^1 (V)$ for all $j \in \N$ and 
\[
 M\left( D (\vec u_j|_B)^{-1} \right) \weakcs M \left( D (\vec u|_B)^{-1} \right) \quad \text{in } \mc{M} (V) 
\text{ as } j \to \infty .
\]
\end{enumerate}
If, in addition, the sequence $\{ \det D (\vec u_j|_B)^{-1} \}_{j \in \N}$ is equiintegrable in $V$, then the convergence in \ref{item:ujINVi}) holds in the weak topology of $W^{1,1} (V, \Rn)$, and the convergence in \ref{item:ujINViii}) holds in the weak topology of $L^1 (V)$.

\item\label{item:thINV4} For a subsequence we have that $\chi_{\imT (\vec u_j, \O)} \to \chi_{\imT (\vec u, \O)}$ a.e.\ and in $L^1 (\Rn)$ as $j \to \infty$.

\end{enumerate}
\end{proposition}

\begin{proof}
 \emph{Part \ref{item:thINV2}):} Let $U$ and $K$ be a set in $\mc{U}^N_{\vec u}$ and a compact subset of $\imT(\vec u, U)$. 
  By Proposition \ref{pr:good_Ut} there exists $\delta>0$ such that for a.e.\ $t\in (0,\delta)$
  $$U_t\in \bigcap_{j\in \N} \mathcal U_{\vec u_j}^N\quad \wedge\quad K\subset \imT(\vec u, U_t).$$
  By the embedding of Proposition \ref{pr:embedding}, 
  the weak continuity of minors of  \cite[Thm.\ 4.11]{BaCuOl81},
  and \cite[Lemma 8.2]{HeMo12}, for a.e.\ such $t$ there exists a subsequence for which
  $$ (\cof D\vec u_j)\vecg\nu_t \weakc (\cof D\vec u)\vecg\nu\quad \text{in }L^1(\partial U_t, \R^n),$$
  where $\vecg\nu_t$ is the unit exterior normal to $U_t$. 
  That $K\subset \imT(\vec u_j, U_t)\subset \imT(\vec u_j, \Omega)$ then follows by 
  Lemma \ref{le:224} and the homotopy-invariance of the degree (as in \cite[Lemma 3.6]{BHM17}).
  \medskip
  
  \emph{Part \ref{item:thINV5}):} The same proof of \cite[Thm.\ 6.3(b)]{BHM17} remains valid.
  It is necessary to take into account that if a map is differentiable at at given point then the condition of regular approximate differentiability,
  used in \cite{BHM17}, is automatically satisfied. Also, the proof uses \cite[Prop.\ 2.6 and Lemma 2.24]{BHM17},
  which have to be replaced by Proposition \ref{pr:fine} and Lemma \ref{le:224} (their Orlicz counterparts).
  \medskip
  
  \emph{Parts \ref{item:thINV3}) and \ref{item:thINV4}):} The proof of \cite[Thm.\ 6.3(c)]{BHM17} remains valid upon replacing
  Proposition 5.3, Equation (5.1), Lemma 2.24, and Lemma 5.18(a) in \cite{BHM17} with Proposition \ref{prop:locINV},
  Equation \eqref{eq:imTUimGU}, Lemma \ref{le:224}, and Proposition \ref{pr:good_Ut} of this paper. 
\end{proof}

\section{Functionals defined in the deformed configuration}\label{se:lower}

Let $W:\R^{n\times n}\to [0,\infty)$ be a polyconvex function. Assume that
\begin{align} \label{eq:coercivity1}
 W(\vec F)\geq cA(|\vec F|) + h(\det \vec F),\quad \vec F\in \R^{n\times n},
\end{align}
for a constant $c>0$ and a Borel function $h:(0,\infty)\to [0,\infty)$ such that
\begin{align} \label{eq:coercivity2}
  \lim_{t\searrow 0} h(t) = \lim_{t\to\infty} \frac{h(t)}{t} = \infty.
\end{align}

\begin{theorem} \label{th:existenceNematic}
 Let $\Omega$ be a Lipschitz domain of $\R^n$, $\Gamma$ an $(n-1)$-rectifiable subset of $\partial \Omega$
 with $\mathcal H^{n-1}(\Gamma)>0$, and $\vec u_0:\Gamma\to\R^n$. Define $\mathcal B$ as the set
 of $(\vec u, \vec n)$ where $\vec u \in \mathcal A$, $\vec u|_{\Gamma}=\vec u_0$
 and $\vec n \in W^{1,2}(\imT(\vec u, \Omega), \Sn)$. 
 Let $W:\R^{n\times n}_+\to [0,\infty)$ be a polyconvex function such that Eqs.\ \eqref{eq:coercivity1}
 and \eqref{eq:coercivity2} hold for a constant $c>0$ and a Borel function $h:(0,\infty)\to[0,\infty)$.
 Define $W_{\mec}$ as in \eqref{eq:model2}.
 If $\mathcal B\ne \vacio$ and 
 \begin{align}
  \label{eq:model3}
 I(\vec u, \vec n)= \int_\Omega W_{\mec}(D\vec u(\vec x), \vec n(\vec u(\vec x)))\dd\vec x 
  + \int_{\imT(\vec u,\Omega)} |D\vec n(\vec y)|^2 \dd\vec y
 \end{align}
 is not identically infinity in $\mathcal B$, 
 then $I$ attains its minimum in $\mathcal B$. 
\end{theorem}

\begin{proof}
 The only substantial difference with the proof of \cite[Thm.\ 8.2]{BHM17}
 is the need of using Proposition \ref{th:ujINV} and equality \eqref{eq:imTUimGU}
 instead of \cite[Thm.\ 6.3 and equality (5.1)]{BHM17} in the proofs of
 \cite[Props.\ 7.1 and 7.8]{BHM17}.
\end{proof}

The other main conclusions in \cite{BHM17} are the lower semicontinuity 
for $\Div$-quasiconvex integrals (under the constraint of incompressibility) 
of Proposition 7.6; the lower semicontinuity for the model for plasticity of \cite{DaFo92,FoGa95paper};
the existence of minimizers in Theorem 8.6 for the Landau-de Gennes model for nematic elastomers of \cite{CaGaL14};
and Theorem 8.9 for the magnetostriction model of \cite{Kruzik15} where minimizers
$(\vec u, \vec m)$ are sought for
$$\int_\Omega W(D\vec u(\vec x), \vec m(\vec u(\vec x)))\dd\vec x 
+ \int_{\imT(\vec u, \Omega)} |D\vec m(\vec y)|^2 \dd\vec y  + \frac{1}{2}\int_{\R^n} |Du_{\vec m}(\vec y)|^2 \dd\vec y,$$
being $u_{\vec m}$ the unique weak solution to Maxwell's equation $$\div (-Du_{\vec m}
+ \chi_{\imT(\vec u, \Omega)}\vec m)=0\ \text{in}\ \R^n.$$
All of these results (not only the existence of minimizers for \eqref{eq:model1}, stated in Theorem \ref{th:existenceNematic})
can be proved under the milder coercivity condition \eqref{eq:L2logL} considered in this paper,
using the results of Sections \ref{se:H1continuity} and \ref{se:classA}.

{\small

  \section*{Acknowledgements}
  
  We are greateful to Carlos Mora-Corral for bringing to our attention the proof by Kauhanen, Koskela \& Mal\'y
  of the Lusin's property satisfied by Orlicz-Sobolev maps. We also thank Stanislav Hencl to whom B.S. has spoken during the conference \lq\lq Methods of Real Analysis and Theory of Elliptic Systems \rq\rq,  Rome.
  The research of D.H.\ and B.S.\ was supported, respectively,
  by the FONDECYT project 1150038 of the Chilean Ministry of Education
  and by University of Naples Project VAriational TECHniques in Advanced MATErials (VATEXMATE). The project has started during the visit of B.S. to Pontificia Universidad Cat\'olica de Chile in July 2018. She would like to thank for the friendly atmosphere during her visit.

  \bibliography{biblio} \bibliographystyle{siam}
}
\Addresses

\end{document}